\documentclass[reqno]{amsart}
\usepackage{}
\usepackage{stmaryrd}
\usepackage{mathrsfs}
\usepackage{cases}
\usepackage{amsfonts}
\usepackage{amssymb}
\usepackage{amsmath}
\usepackage{tikz}
\newtheorem{theorem}{Theorem}[section]
\newtheorem{lemma}[theorem]{Lemma}
\newtheorem{proposition}[theorem]{Proposition}

\theoremstyle{definition}
\newtheorem{definition}[theorem]{Definition}

\newtheorem{remark}[theorem]{Remark}
\numberwithin{equation}{section}

\newcommand{\blankbox}[2]

\begin{document}
\title[Inclusion relations between modulation and Triebel-Lizorkin spaces]
{Inclusion relations between modulation and Triebel-Lizorkin spaces$^*$}
\author{WEICHAO GUO}
\address{School of Mathematical Sciences, Xiamen University,
Xiamen, 361005, P.R. China} \email{weichaoguomath@gmail.com}
\author{HUOXIONG WU}
\address{School of Mathematical Sciences, Xiamen University,
Xiamen, 361005, P.R. China} \email{huoxwu@xmu.edu.cn}
\author{GUOPING ZHAO}
\address{School of Applied Mathematics, Xiamen University of Technology,
Xiamen, 361024, P.R.China} \email{guopingzhaomath@gmail.com}
\thanks{$^*$ Partly supported by the NNSF of China (Grant Nos. 11371295, 11471041) and the NSF of Fujian Province of China (No. 2015J01025).}

\begin{abstract}
In this paper, we obtain the sharp conditions of the inclusion relations between modulation spaces $M_{p,q}^s$ and Triebel-Lizorkin spaces $F_{p,r}$ for $p\leq 1$, which greatly improve and extend the results for the embedding relations between local Hardy spaces and modulation spaces obtained by Kobayashi, Miyachi and Tomita in [Studia Math. 192 (2009), 79-96].
\end{abstract}

\maketitle

\section{INTRODUCTION}
The modulation spaces $M_{p,q}^{s}$ were introduced by Feichtinger\cite{Feichtinger} in 1983 by the short-time Fourier transform.
One can find some motivations and basic properties in \cite{Feichtinger_Survey}.
Modulation spaces have a close relationship with the topics of time-frequency analysis (see \cite{Grochenig}),
and that they has been regarded as a appropriate spaces for the study of PDE (see \cite{RSW_2012}).
As function spaces associated with the uniform decomposition, modulation spaces have many beautiful properties,
for instance, the properties of product and convolution on modulation spaces (see \cite{Feichtinger_Banach convolution, Guo_Characterization, Guo_SCI China}),
the boundedness of unimodular multipliers on modulation spaces (see \cite{AKKL_JFA_2007, Miyachi, Zhao_Guo_nolinear analysis}).
One can also see \cite{Guo_Chen, WZG_JFA_2006, Wang_book, WH_JDE_2007} for the study of nonlinear evolution equations on modulation spaces.

A basic but important problem is to find the relations between modulation spaces and classical function spaces.
Many authors have paid their attentions to this topic, for example, one can see Gr\"{o}bner \cite{Grober_thesis}, Okoudjou \cite{K.Okoudjou_embedding_modulation_Classical}, Toft \cite{Toft_embedding,Toft_Continunity},
Sugimoto-Tomita \cite{Sugimoto_Tomita}, Masaharu-Sugimoto \cite{sobolev and modulation} and Wang-Huang \cite{Wang_Huang_JDE}.
In particular, Kobayashi-Miyachi-Tomita \cite{Lcal hardy and modulation_studia} studied the inclusion relations between local Hardy spaces $h_p$ and modulation spaces $M_{p,q}^s$ and established the following results.
\medskip

\hspace{-12pt}\textbf{Theorem A} (cf. \cite{Lcal hardy and modulation_studia})\quad
Let $0<p\leq 1$, $0<q\leq \infty$ and $s\in \mathbb{R}$.
Then, $M_{p,q}^{s} \subset h_{p}$
  if and only if one of the following conditions holds:
    \begin{enumerate}
     \item
     $1/p\leq 1/q, s\geq 0$;
     \item
     $1/p>1/q, s>n(1/p-1/q)$.
   \end{enumerate}
~\\
\textbf{Theorem B} (cf. \cite{Lcal hardy and modulation_studia})\quad
Let $0<p\leq 1$, $0<q\leq \infty$ and $s\in \mathbb{R}$.
Then, $h_{p} \subset M_{p,q}^s$
  if and only if one of the following conditions holds:
    \begin{enumerate}
     \item
     $1/q\leq 1/p, s\leq n(1-1/p-1/q)$;
     \item
     $1/q>1/p, s<n(1-1/p-1/q)$.
   \end{enumerate}

We recall that the local Hardy spaces $h_p\, (0<p<\infty)$, which was introduced by by Goldberg in \cite{Goldberg}, is equivalent with the inhomogeneous
Triebel-Lizorkin space $F_{p,r}$ with $r=2$. Note that the Triebel-Lizorkin spaces contain many important function spaces, such as Lebesgue spaces, Sobolev spaces, and Hardy spaces et al. It is natural to ask whether the inclusion relations above in Theorems A and B still hold if replaced $F_{p,2}$ by the general $F_{p,r}$ for $r>0$? This question will be addressed by our next theorems.

\begin{theorem}\label{theorem 1}
Let $0<p\leq 1$, $0<q,r\leq \infty$ and $s\in \mathbb{R}$.
Then, $M_{p,q}^{s} \subset F_{p,r}$
  if and only if one of the following conditions holds:
    \begin{enumerate}
     \item
     $1/p\leq 1/q, s\geq 0,  1/r\leq 1/q$;
     \item
     $1/p>1/q , s>n(1/p-1/q)$.
   \end{enumerate}
\end{theorem}

\begin{theorem}\label{theorem 2}
Let $0<p\leq 1$, $0<q,r\leq \infty$ and $s\in \mathbb{R}$.
Then, $F_{p,r} \subset M_{p,q}^s$
  if and only if one of the following conditions holds:
    \begin{enumerate}
     \item
     $1/q\leq 1/p, s\leq n(1-1/p-1/q)$;
     \item
     $1/q>1/p, s<n(1-1/p-1/q)$.
   \end{enumerate}
\end{theorem}

\begin{remark} To make the reader more clearly understand our work in this paper,
we would like to make a comparison between our theorems and the main results in \cite{Lcal hardy and modulation_studia}, i.e. Theorem A and B.

Cleary, Theorems \ref{theorem 1} and \ref{theorem 2} in this paper are great improvement and extension of Theorems A and B. Also, in the technical level,
we have to work under the framework of Triebel-Lizorkin spaces, which is a lack of such a simple form of atoms as in $h_p$, and then the most important problem we face is how to estimate the more complicated atoms of Triebel-Lizorkin spaces in modulation spaces. In addition, we drop the method used in \cite{Lcal hardy and modulation_studia}, which is deeply depend on an equivalent norm of modulation space (see Lemma 2.2 in \cite{Lcal hardy and modulation_studia}),
and use the "standard" norm of modulation spaces. We would like to give an independent proof (quite different from \cite{Lcal hardy and modulation_studia}) on both sufficiency and necessity parts, which seems more efficient and clear than those in \cite{Lcal hardy and modulation_studia}.
\end{remark}

This paper is organized as follows. In Section 2, we will recall some basic notations and definitions, and present some preliminary lemmas, which will be used in our proofs. The proofs of our main results will be given in Section 3. Also, we will end this paper by presenting a open problem and its difficulties.

\section{PRELIMINARIES}

In this section, we first recall some notations. Let $C$ be a positive constant that may depend on $n,p,q,r,s.$
The notation $X\lesssim  Y$ denotes the statement that $X\leq CY$,
the notation $X\sim Y$ means the statement $X\lesssim Y \lesssim X$,
and the notation $X\simeq Y$ denotes the statement $X=CY$.
For a multi-index $k=(k_1,k_2,...k_n)\in \mathbb{Z}^{n}$,
we denote $|k|_{\infty}: =\mathop{sup}_{i=1,2...n}|k_i|$, and $\langle k\rangle: =(1+|k|^{2})^{{1}/{2}}.$

Let $\mathscr {S}:= \mathscr {S}(\mathbb{R}^{n})$ be the Schwartz space
and $\mathscr {S}':=\mathscr {S}'(\mathbb{R}^{n})$ be the space of tempered distributions.
We define the Fourier transform $\mathscr {F}f$ and the inverse Fourier transform $\mathscr {F}^{-1}f$ of $f\in \mathscr {S}(\mathbb{R}^{n})$ by
$$
\mathscr {F}f(\xi)=\hat{f}(\xi)=\int_{\mathbb{R}^{n}}f(x)e^{-2\pi ix\cdot \xi}dx
~~
,
~~
\mathscr {F}^{-1}f(x)=\hat{f}(-x)=\int_{\mathbb{R}^{n}}f(\xi)e^{2\pi ix\cdot \xi}d\xi.
$$

We recall some definitions of the function spaces treated in this paper.
\begin{definition}
Let $s\in \mathbb{R}, 0<p,q\leq \infty$. The weighted Lebesgue space $L_{x, p}^s$ consists of all measurable functions $f$ such that
\begin{numcases}{\|f\|_{L_{x, p}^s}=}
     \left(\int_{\mathbb{R}^n}|f(x)|^p \langle x\rangle^{ps} dx\right)^{{1}/{p}}, &$p<\infty$   \\
     ess\sup_{x\in \mathbb{R}^n}|f(x)\langle x\rangle^s|,  &$p=\infty$
\end{numcases}
is finite.
If $f$ is defined on $\mathbb{Z}^n$, we denote
\begin{numcases}{\|f\|_{l_{k,p}^{s,0}}=}
\left(\sum_{k\in \mathbb{Z}^n}|f(k)|^p \langle k\rangle^{ps}\right)^{{1}/{p}}, &$p<\infty$
\\
\sup_{k\in \mathbb{Z}^n}|f(k)\langle k\rangle^s|,\hspace{15mm} &$p=\infty$
\end{numcases}
and $l_{k,p}^{s,0}$ as the (quasi) Banach space of functions $f: \mathbb{Z}^n\rightarrow \mathbb{C}$ whose $l_{k,p}^{s,0}$ norm is finite.
If $f$ is defined on $\mathbb{N}=\mathbb{Z}^+\cup \{0\}$, we denote
\begin{numcases}{\|f\|_{l_{j,p}^{s,1}}=}
\left(\sum_{j\in \mathbb{N}}2^{jsp}|f(j)|^p\right)^{{1}/{p}}, &$p<\infty$
\\
\sup_{j\in \mathbb{N}}|2^{js}f(j)|,\hspace{15mm} &$p=\infty$
\end{numcases}
and $l_{j,p}^{s,1}$ as the (quasi) Banach space of functions $f: \mathbb{N}\rightarrow \mathbb{C}$ whose $l_{j,p}^{s,1}$ norm is finite.
We write $L_p^s, l_p^{s,0}, l_p^{s,1}$ for short, respectively, if there is no confusion.
\end{definition}

The translation operator is defined as $T_{x_0}f(x)=f(x-x_0)$ and
the modulation operator is defined as $M_{\xi}f(x)=e^{2\pi i\xi \cdot x}f(x)$, for $x, x_0, \xi\in\mathbb{R}^n$.
Fixed a nonzero function $\phi\in \mathscr{S}$, the short-time Fourier
transform of $f\in \mathscr{S}'$ with respect to the window $\phi$ is given by
\begin{equation}
V_{\phi}f(x,\xi)=\langle f,M_{\xi}T_x\phi\rangle,
\end{equation}
and that can written as
\begin{equation}
  V_{\phi}f(x,\xi)=\int_{\mathbb{R}^n}f(y)\overline{\phi(y-x)}e^{-2\pi iy\cdot \xi}dy
\end{equation}
if $f\in \mathscr{S}$.
We give the (continuous) definition of modulation space $\mathcal {M}_{p,q}^s$ as follows.

\begin{definition}
Let $s, t \in \mathbb{R}$, $0<p,q\leq \infty$. The (weighted) modulation space $\mathcal {M}_{p,q}^s$ consists
of all $f\in \mathscr{S}'(\mathbb{R}^n)$ such that the (weighted) modulation space norm
\begin{equation}
\begin{split}
\|f\|_{\mathcal {M}_{p,q}^{s}}&=\big\|\|V_{\phi}f(x,\xi)\|_{L_{x,p}}\big\|_{L_{\xi,q}^s}
\\&
=\left(\int_{\mathbb{R}^n}\left(\int_{\mathbb{R}^n}|V_{\phi}f(x,\xi)|^{p}\langle x\rangle^{tp}dx\right)^{{q}/{p}}\langle \xi\rangle^{sq}dx\right)^{{1}/{q}}
\end{split}
\end{equation}
is finite, with the usual modifications when $p=\infty$ or $q=\infty$.
This definition is independent of the choice of the window $\phi\in \mathscr{S}$.
\end{definition}
Applying the frequency-uniform localization techniques, one can give an alternative definition of modulation spaces (see \cite{Triebel_modulation space, WH_JDE_2007} for details).

We denote by $Q_{k}$ the unit cube with the center at $k$. Then the family $\{Q_{k}\}_{k\in\mathbb{Z}^{n}}$
constitutes a decomposition of $\mathbb{R}^{n}$.
Let $\rho \in \mathscr {S}(\mathbb{R}^{n}),$
$\rho: \mathbb{R}^{n} \rightarrow [0,1]$ be a smooth function satisfying that $\rho(\xi)=1$ for
$|\xi|_{\infty}\leq {1}/{2}$ and $\rho(\xi)=0$ for $|\xi|\geq 3/4$. Let
\begin{equation}
\rho_{k}(\xi)=\rho(\xi-k),  k\in \mathbb{Z}^{n}
\end{equation}
be a translation of \ $\rho$.
Since $\rho_{k}(\xi)=1$ in $Q_{k}$, we have that $\sum_{k\in\mathbb{Z}^{n}}\rho_{k}(\xi)\geq1$
for all $\xi\in\mathbb{R}^{n}$. Denote
\begin{equation}
\sigma_{k}(\xi)=\rho_{k}(\xi)\left(\sum_{l\in\mathbb{Z}^{n}}\rho_{l}(\xi)\right)^{-1},  ~~~~ k\in\mathbb{Z}^{n}.
\end{equation}
It is easy to know that $\{\sigma_{k}(\xi)\}_{k\in\mathbb{Z}^{n}}$
constitutes a smooth decomposition of $\mathbb{R}^{n}$, and $%
\sigma_{k}(\xi)=\sigma(\xi-k)$. The frequency-uniform decomposition
operators can be defined by
\begin{equation}
\Box_{k}:= \mathscr{F}^{-1}\sigma_{k}\mathscr{F}
\end{equation}
for $k\in \mathbb{Z}^{n}$.
Now, we give the (discrete) definition of modulation space $M_{p,q}^s$.

\begin{definition}
Let $s\in \mathbb{R}, 0<p,q\leq \infty$. The modulation space $M_{p,q}^s$ consists of all $f\in \mathscr{S}'$ such that the norm
\begin{equation}
\|f\|_{M_{p,q}^s}:=\left( \sum_{k\in \mathbb{Z}^{n}}\langle k\rangle ^{sq}\|\Box_k f\|_{p}^{q}\right)^{1/q}
\end{equation}
is finite. We write $M_{p,q}:=M_{p,q}^0$ for short.
\end{definition}

\begin{remark}
We remark that the above definition is independent of the choice of $\sigma$ (see \cite{Wang_book}).
So, we can use appropriate $\sigma$ according to our need.
In the definition above, the function sequence  $\{\sigma_k(\xi)\}_{k\in \mathbb{Z}^n}$ satisfies
$\sigma_k(\xi)=1$ and $\sigma_k(\xi)\sigma_l(\xi)=1$ if $k\neq l$, where $|\xi|_{\infty}\leq 1/4$.
We also recall that the definition of $\mathcal {M}_{p,q}^s$ and $M_{p,q}^s$ are equivalent.
In this paper, we mainly use the discrete definition of modulation space, i.e., $M_{p,q}^s$.
\end{remark}

Next, we recall function spaces associated with the dyadic decomposition of $\mathbb{R}^{n}$.
Let $\varphi(\xi)$ be a smooth bump function supported in the ball $\{\xi: |\xi|<\frac{3}{2}\}$ and be equal to 1 on the ball $\{\xi: |\xi|\leq \frac{4}{3}\}$.
Denote
\begin{equation}
\psi(\xi)=\varphi(\xi)-\varphi(2\xi),
\end{equation}
and a function sequence
\begin{equation}
\begin{cases}
\psi_j(\xi)=\psi(2^{-j}\xi),~j\in \mathbb{Z}^{+},
\\
\psi_0(\xi)=1-\sum_{j\in \mathbb{Z}^+}\psi_j(\xi)=\varphi(\xi).
\end{cases}
\end{equation}
For integers $j\in \mathbb{N}$, we define the Littlewood-Paley operators
\begin{equation}
\Delta_j=\mathscr{F}^{-1}\psi_j\mathscr{F}.
\end{equation}
Let $0< p,q\leq\infty$ and $s\in \mathbb{R}$. For $f\in\mathscr {S}'$, set
\begin{equation}
\|f\|_{B_{p,q}^s}=\left(\sum_{j=0}^{\infty}2^{jsq}\|\Delta_jf\|_{L_p}^q \right)^{1/q}.
\end{equation}
The (inhomogeneous) Besov space is the space of all tempered distributions $f$ for which the quantity $\|f\|_{B_{p,q}^s}$ is finite.

Let $0<p<\infty$, $0< q\leq\infty$ and $s\in \mathbb{R}$. For $f\in\mathscr {S}'$, set
\begin{equation}
\|f\|_{F_{p,q}^s}=\left\|\left(\sum_{j=0}^{\infty}2^{jsq}|\Delta_jf|^q \right)^{1/q}\right\|_{L_p}.
\end{equation}
The (inhomogeneous) Triebel-Lizorkin space is the space of all tempered distributions $f$ for which the quantity $\|f\|_{F_{p,q}^s}$ is finite.

\bigskip
Let $\mathbb{Q}^n$ be the collection of all cubes $Q_{\nu,k}$ in $\mathbb{R}^n$ with sides parallel to the axes,
centered at $2^{-\nu}k$, and with side length $l(Q_{\nu,k})=2^{-\nu}$, where $k\in \mathbb{Z}^n$, $\nu\in \mathbb{N}$.

Let $r>0$, we use $rQ$ to denote the cube in $\mathbb{R}^n$ concentric with $Q$ satisfying $l(rQ)=rl(Q)$.
We write $(\nu,k)<(\nu',k')$ if $\nu\geq \nu'$ and
\begin{equation}
  Q_{\nu,k}\subset 2 Q_{\nu',k'}\hspace{3mm} \text{with}\hspace{3mm}Q_{\nu,k}, Q_{\nu',k'}\in \mathbb{Q}^n.
\end{equation}
For $c\in \mathbb{R}$, we denote $c_{+}=\max\{c, 0\}$, and use $[c]$ to represent the largest integer less than or equal to $c$.
We recall some definitions about $s$-atom and $(Q,s,p,q)$-atom, which are very useful in our proofs.

\begin{definition}\label{definition, atom}
Let $0<p\leq 1<q\leq \infty$, $s\in \mathbb{R}$.
Let $K$ and $L$ be the integers with
\begin{equation}\label{for proof, definition of atom, 2}
  K\geq ([s]+1)_{+}\hspace{6mm}and\hspace{6mm}L\geq \max\{[n(1/p-1)_{+}-s],1\}
\end{equation}
    \begin{enumerate}
     \item
     A (complex-valued) function $a(x)$ is called a $s$-atom if
     \begin{equation}\label{for proof, definition of atom, 1}
       \textbf{supp}a\subset 5Q
     \end{equation}
     for some $Q=Q_{0k}\in \mathbb{Q}^n$ and
     \begin{equation}
       |D^{\alpha}a(x)|\leq 1\hspace{6mm}\text{for}\hspace{6mm}|\alpha|\leq K.
     \end{equation}
     \item
     Let $Q=Q_{\nu,k}\in \mathbb{Q}^n$. The function $a(x)$ is called a $(Q,s,p,q)$-atom if
     (\ref{for proof, definition of atom, 1}) is satisfied,
     \begin{equation}
       |D^{\alpha}a(x)|\leq |Q|^{-1/q+s/n-|\alpha|/n}\hspace{6mm}\text{for}\hspace{6mm}|\alpha|\leq K,
     \end{equation}
     and
     \begin{equation}
       \int_{\mathbb{R}^n}x^{\beta}a(x)dx=0\hspace{6mm}\text{for}\hspace{6mm}|\beta|\leq L.
     \end{equation}
     \item
     The distribution $g\in \mathscr{S}'$ is called an $(s,p,q)$-atom if
     \begin{equation}
       g=\sum_{(\mu,l)<(\nu,k)}d_{\mu l}a_{\mu l}(x)\hspace{10mm}(\text{convergence\ in\ }F_{p,q}^s)
     \end{equation}
     for some $\nu\in \mathbb{N}$ and $k\in \mathbb{Z}^n$, where $a_{\mu l}$ is a $(Q_{\mu l},s,p,q)$-atom and $d_{\mu l}$ are complex numbers with
     \begin{equation}
       \left(\sum_{(\mu,l)<(\nu,k)}|d_{\mu l}|^q\right)^{1/q}\leq |Q_{\nu k}|^{1/q-1/p}
     \end{equation}
     with usual modification if $q=\infty$.
   \end{enumerate}
\end{definition}

We also recall the atomic decomposition for $F_{p,q}^s$. One can find following lemma and its historic remarks in Triebel's book \cite[Section 3.2]{Triebel_book_1992}.
\begin{lemma}\label{lemma, atom}
  Let $0<p\leq 1<q\leq \infty$, $s\in \mathbb{R}$. Let $K$ and $L$ be fixed integers satisfying (\ref{for proof, definition of atom, 2}).
  A distribution $f\in \mathscr{S}'$ is an element of $F_{p,q}^s$ if and only if it can be represented as
  \begin{equation}
    \sum_{j=1}^{\infty}(\mu_ja_j+\lambda_jg_j)\hspace{10mm}(\text{convergence\ in\ }\mathscr{S}')
  \end{equation}
  where $a_j$ are $s$-atoms, $g_j$ are $(s,p,q)$-atoms, $\mu_j$ and $\lambda_j$ are complex numbers with
  \begin{equation}\label{for proof, atom, 1}
    \left(\sum_{j=1}^{\infty}|\mu_j|^p+|\lambda_j|^p\right)^{1/p}\lesssim \|f\|_{F_{p,q}^s}.
  \end{equation}
\end{lemma}

We recall a inclusion relations between modulation and Besov spaces.

\begin{lemma}[see \cite{Wang_Huang_JDE}]\label{lemma, modulation and Besov spaces}
  Let $0<p,q\leq \infty$. Then the following two statements is true.
  \begin{enumerate}
    \item
    $M_{p,q}^s\subset B_{p,q} \Longleftrightarrow s\geq 0\vee [n(1/p-1/q)]\vee [n(1-1/p-1/q)]$;
    \item
    $B_{p,q}\subset M_{p,q}^s \Longleftrightarrow s\leq 0\wedge [n(1/p-1/q)]\wedge [n(1-1/p-1/q)]$.
  \end{enumerate}
\end{lemma}

The following Bernstein multiplier theorem will be used in our proof.

\begin{lemma}[Bernstein multiplier theorem]
Let $0<p\leq 1$, $\partial^{\gamma}f\in L^2$ for $|\gamma|\leq [n(1/p-1/2)]+1$. Then,
\begin{equation}
  \|\mathscr{F}^{-1}f\|_{L_p}\lesssim \sum_{|\gamma|\leq [n(1/p-1/2)]+1}\|\partial^{\gamma}f\|_{L_2}.
\end{equation}
\end{lemma}

We recall a convolution inequality for $p\leq 1$.
\begin{lemma}[{Weighted convolution in $L_{p}$ with $p<1$}, see \cite{Triebel_book_1983}]\label{lemma, convolution for p<1}
Let $0<p< 1$,  $B(x_0,R)=\{x: |x-x_0|\le R\}$.
Suppose $f, g\in L^p$ with Fourier support in $B(x_0,R)$ and $B(x_1,R)$ respectively. Then there exists a constant $C>0$ which is independent of $x_0, x_1, R, f$ such that
$$
\||f|*|g|\|_{L_p} \le CR^{n(1/p-1)} \|f\|_{L_p} \|g\|_{L_p}.
$$
\end{lemma}

Finally, we recall some embedding lemmas.
\begin{lemma}[Sharpness of embedding, for uniform decomposition] \label{lemma, Sharpness of embedding, for uniform decomposition}
Suppose $0<q_1,q_2\leq \infty$, $s_1,s_2\in \mathbb{R}$. Then
\begin{equation}
l_{q_1}^{s_1,0}\subset l_{q_2}^{s_2,0}
\end{equation}
holds if and only if
\begin{equation}
\begin{cases}
s_2\leq s_1  \\
\frac{1}{q_2}+\frac{s_2}{n}< \frac{1}{q_1}+\frac{s_1}{n}
\end{cases}
\text{or} \hspace{10mm}
\begin{cases}
s_2=s_1\\
q_2=q_1.
\end{cases}
\end{equation}
\end{lemma}

\begin{lemma}[Sharpness of embedding, for dyadic decomposition] \label{lemma, Sharpness of embedding, for dyadic decomposition}
Suppose $0<q_1,q_2\leq \infty$, $s_1,s_2\in \mathbb{R}$. Then
\begin{equation}
l_{q_1}^{s_1,1}\subset l_{q_2}^{s_2,1}
\end{equation}
holds if and only if
\begin{equation}
s_2<s_1
\hspace{10mm}
\text{or} \hspace{10mm}
\begin{cases}
s_2=s_1\\
1/q_2\leq 1/q_1.
\end{cases}
\end{equation}
\end{lemma}

\section{Proofs of Theorems \ref{theorem 1} and \ref{theorem 2}}

\begin{proposition}\label{proposition, necessity, 1}
  Let $0<p< \infty$, $0<q,r\leq \infty$, $s\in \mathbb{R}$.
  Then we have
  \begin{enumerate}
  \item
  $M_{p,q}^s\subset F_{p,r} \Longrightarrow l_q^{s,0}\subset l_p^{0,0}$,
  \item
  $F_{p,r}\subset M_{p,q}^s \Longrightarrow l_p^{0,0}\subset l_q^{s,0}$.
\end{enumerate}
\end{proposition}
\begin{proof}
Take $g$ to be smooth function with compact Fourier support on $B(0,10^{-10})$, $\widehat{g_k}=\widehat{g}(\cdot-k)$.
Denote
\begin{equation}
  G_N=\sum_{k\in \mathbb{Z}^n}a_kT_{Nk}g_k,
\end{equation}
where $\{a_k\}_{k\in \mathbb{Z}^n}$ is a truncated (only finite nonzero items) sequence of nonnegative real number,
$T$ is the translation operator,
$N$ is some large integer to be chosen later.
We have $\Box_kG_N=a_kT_{Nk}g_k$. In addition, we have $\|g_k\|_{F_{p,r}}\sim \|g_k\|_{L_p}$.

By the definition of modulation space $M_{p,q}$, we have
\begin{equation}
  \begin{split}
    \|G_N\|_{M_{p,q}^s}
    \sim
    \|\{\|\Box_kG_N\|_{L_p}\}\|_{l_q^{s,0}}
   & \sim
    \|\{a_k\|T_{Nk}g_k\|_{L_p}\}\|_{l_q^{s,0}}
    \\
    \sim
    \|\{a_kg\|_{L_p}\}\|_{l_q^{s,0}}
    \sim &
    \|\{a_k\}\|_{l_q^{s,0}}.
  \end{split}
\end{equation}
On the other hand, letting $N\rightarrow \infty$,
we use the almost orthogonality of $\{a_kT_{Nk}g_k\}_{k\in \mathbb{Z}^n}$ to deduce that
\begin{equation}
  \begin{split}
    \lim_{N\rightarrow \infty}\|G_N\|_{F_{p,r}}
    = &
    \lim_{N\rightarrow \infty}\left(\sum_{k\in \mathbb{Z}^n}a_k^p\|T_{Nk}g_k\|^p_{F_{p,r}}\right)^{1/p}
    \\
    \sim &
    \left(\sum_{k\in \mathbb{Z}^n}a_k^p\|g_k\|^p_{L_p}\right)^{1/p}
    \sim
    \left(\sum_{k\in \mathbb{Z}^n}a_k^p\right)^{1/p}
    \sim
    \|\{a_k\}\|_{l_p^{0,0}}.
  \end{split}
\end{equation}
Thus, if $M_{p,q}^s\subset F_{p,r}$, we deduce that
\begin{equation}
  \|\{a_k\}\|_{l_p^{0,0}} \sim \lim_{N\rightarrow \infty}\|G_N\|_{F_{p,r}}\lesssim \lim_{N\rightarrow \infty}\|G_N\|_{M_{p,q}^s}\sim \|\{a_k\}\|_{l_q^{s,0}}
\end{equation}
for any truncated sequence $\{a_k\}_{k\in \mathbb{Z}}$, which implies the desired inclusion relation $l_q^{s,0}\subset l_p^{0,0}$.

On the other hand, if $F_{p,r}\subset M_{p,q}^s$ holds, we conclude
\begin{equation}
  \|\{a_k\}\|_{l_q^{s,0}} \sim \lim_{N\rightarrow \infty}\|G_N\|_{M_{p,q}^s} \lesssim \lim_{N\rightarrow \infty}\|G_N\|_{F_{p,r}}\sim \|\{a_k\}\|_{l_p^{0,0}},
\end{equation}
which implies $l_p^{0,0}\subset l_q^{s,0}$.
\end{proof}

\begin{proposition}\label{proposition, necessity, 2}
  Let $0<p< \infty$, $0<q,r\leq \infty$, $s\in \mathbb{R}$.   Then we have
  \begin{enumerate}
  \item
  $M_{p,q}^s\subset F_{p,r} \Longrightarrow l_q^{s+n/q,1}\subset l_p^{n(1-1/p),1}$,
  \item
  $F_{p,r}\subset M_{p,q}^s \Longrightarrow l_p^{n(1-1/p),1}\subset l_q^{s+n/q,1}$.
\end{enumerate}
\end{proposition}
\begin{proof}
  Choose a smooth function $h$ with compact Fourier support on $3/4\leq |\xi|\leq 4/3$,
satisfying $\hat{h}(\xi)=1$ on $7/8\leq |\xi|\leq 8/7$.
Denote $\widehat{h_j}(\xi)=\widehat{h}(\xi/2^j)$,
\begin{equation}
  \Gamma_j=\{k\in \mathbb{Z}^n: \Box_kh_j\neq 0 \},\quad\, \widetilde{\Gamma_j}=\{k\in \mathbb{Z}^n: \Box_kh_j=\mathscr{F}^{-1}\sigma_k \}.
\end{equation}
Obviously, we have $|\Gamma_j|\sim |\widetilde{\Gamma_j}|\sim 2^{jn}$ for $|j|\geq J$, where $J$ is a sufficient large number.
In addition, $k\sim 2^j$ for $k\in \Gamma_j$ or $k\in \widetilde{\Gamma_j}$.

Denote
$$F_N=\sum_{j\geq J}^{\infty}b_jT_{Nje_0}h_j,$$ where $\{b_j\}_{j=0}^{\infty}$ is a truncated (only finite nonzero items) sequence of nonnegative real number,
$e_0=(1,0,\cdots,0)$ is the unit vector of $\mathbb{R}^n$, $N$ is a sufficient large number to be chosen later.

We first estimate the norm of $M_{p,q}^s$.
Using lemma \ref{lemma, convolution for p<1}, we deduce that
\begin{equation}
  \|\Box_kh_j\|\lesssim 2^{jn(1/p-1)}\|h_j\|_{L_p}\cdot \|\mathscr{F}^{-1}\sigma_k\|_{L_p}\lesssim 1.
\end{equation}
Thus,
\begin{equation}
  \begin{split}
    \|F_N\|_{M_{p,q}^s}
    = &
    \left(\sum_{k\in \mathbb{Z}^n}\langle k\rangle^{sq}\|\Box_kF_N\|_{L_p}^q\right)^{1/q}
    \\
    \leq &
    \left(\sum_{j\geq J}\sum_{k\in \Gamma_j}\langle k\rangle^{sq}b_j^q\|\Box_kh_j\|_{L_p}^q\right)^{1/q}
    \\
    \lesssim &
    \left(\sum_{j\geq J}2^{jsq}b_j^q|\Gamma_j|\right)^{1/q}
    \lesssim
    \left(\sum_{j\geq J}2^{jsq}b_j^q2^{jn}\right)^{1/q}\sim \|\{b_j\}_{j\geq J}\|_{l_q^{s+n/q,1}}.
  \end{split}
\end{equation}
On the other hand,
\begin{equation}
  \begin{split}
    \|F_N\|_{M_{p,q}^s}
    \geq &
    \left(\sum_{j\geq J}\sum_{k\in \widetilde{\Gamma_j}}\langle k\rangle^{sq}b_j^q\|\Box_kh_j\|_{L_p}^q\right)^{1/q}
    \\
    =&
    \left(\sum_{j\geq J}\sum_{k\in \widetilde{\Gamma_j}}\langle k\rangle^{sq}b_j^q\|\mathscr{F}^{-1}\sigma_k\|_{L_p}^q\right)^{1/q}
    \\
    \sim &
    \left(\sum_{j\geq J}2^{jsq}b_j^q|\widetilde{\Gamma_j}|\right)^{1/q}
    \sim
    \left(\sum_{j\geq J}2^{jsq}b_j^q2^{jn}\right)^{1/q}\sim \|\{b_j\}_{j\geq J}\|_{l_q^{s+n/q,1}}.
  \end{split}
\end{equation}
Hence,
\begin{equation}
  \|F_N\|_{M_{p,q}^s}\sim \|\{b_j\}_{j\geq J}\|_{l_p^{n(1-1/p),1}}.
\end{equation}

Now, we turn to the estimate of $\|F_N\|_{F_{p,r}}$.
By the assumption of $h$, we have
$\Delta_jh_j=h_j$, where $\widehat{h_j}(\xi)=\hat{h}(\xi/2^j)$.
By the definition of $F_{p,r}$, we have
\begin{equation}
    \|F_N\|_{F_{p,r}} =
    \|\|\{\Delta_j F_N\}_{j\in \mathbb{N}}\|_{l_r^{0,1}}\|_{L_p}
     =
    \|\|\{b_j h_j\}_{j\geq J}\|_{l_r^{0,1}}\|_{L_p}.
\end{equation}
Using the almost orthogonality of $\{a_jT_{Nje_0}h_j\}_{j=0}^{\infty}$ as $N\rightarrow \infty$, we deduce
\begin{equation}
  \begin{split}
    \lim_{N\rightarrow \infty}\|F_N\|_{F_{p,r}}
    = &
    \left(\sum_{j\geq J}b_j^p\|h_j\|_{L_p}^p\right)^{1/p}
    \\
    \sim &
    \left(\sum_{j\geq J}b_j^p2^{jn(1-1/p)p}\right)^{1/p}\sim \|\{b_j\}_{j\geq J}\|_{l_p^{n(1-1/p),1}}.
  \end{split}
\end{equation}

If $M_{p,q}^s\subset F_{p,r}$ holds, we have
\begin{equation}
  \|F_N\|_{F_{p,r}}\lesssim \|F_N\|_{M_{p,q}^s}.
\end{equation}
Letting $N\rightarrow \infty$, we use the estimates of $\|F_N\|_{F_{p,r}}$ and $\|F_N\|_{M_{p,q}^s}$ obtained above to deduce that
\begin{equation}
  \|\{b_j\}_{j\geq J}\|_{l_p^{n(1-1/p),1}}\lesssim \|\{b_j\}_{j\geq J}\|_{l_q^{s+n/q,1}}
\end{equation}
for all truncated sequence $\{b_j\}_{j\in \mathbb{N}}$, which implies the desired inclusion relation $l_q^{s+n/q,1}\subset l_p^{n(1-1/p),1}$.

Similarly, if $F_{p,r}\subset M_{p,q}^s$ holds, we deduce
\begin{equation}
  \|\{b_j\}_{j\geq J}\|_{l_q^{s+n/q,1}}\lesssim |\{b_j\}_{j\geq J}\|_{l_p^{n(1-1/p),1}}
\end{equation}
for all truncated sequence $\{b_j\}_{j\in \mathbb{N}}$, which implies $l_p^{n(1-1/p),1}\subset l_q^{s+n/q,1}$.
\end{proof}

\begin{proposition}\label{proposition, necessity, special}
  Let $0<p< \infty$, $0<q,r\leq \infty$. Then we have
  $$
  M_{p,q}\subset F_{p,r} \Longrightarrow l_q^{0,1}\subset l_r^{0,1}.
  $$
\end{proposition}
\begin{proof}
  Take $h$ to be a nonzero smooth function with Fourier support $B(0,10^{-10})$.
  We choose a sequence $\{\xi_j\}_{j=0}^{\infty}$ of $\mathbb{Z}^n$ such that $\Delta_jh_j=h_j$ and $\Delta_lh_j=0$ for $j\in \mathbb{N}$ and $i\neq j$,
  where $\widehat{h_j}(\cdot)=\widehat{h}(\cdot-\xi_j)$. Obviously, we have $\xi_j\sim 2^j$.
  Denote
  \begin{equation}
    H=\sum_{j=0}^na_jh_j,
  \end{equation}
  where $\{a_j\}_{j=0}^{\infty}$ is a truncated sequence of nonnegative real number.
  We have
  \begin{equation}
    \|H\|_{F_{p,r}}=\left\|\left(\sum_{j=0}^{\infty}|a_j|^{r}|h_j|^r\right)^{1/r}\right\|_{L_p}
    =
    \left\|\left(\sum_{j=0}^{\infty}|a_j|^{r}|h|^r\right)^{1/r}\right\|_{L_p}
    \sim
    \|\{a_j\}\|_{l_r^{0,1}}.
  \end{equation}
  On the other hand,
  \begin{equation}
  \begin{split}
    \|H\|_{M_{p,q}}=\left(\sum_{k\in \mathbb{Z}^n}\|\Box_kH\|^q_{L_p}\right)^{1/q}
    \sim &
    \left(\sum_{j=0}^{\infty}|a_j|^{q}\|h_j\|_{L_p}^q\right)^{1/q}
    \\
    \sim &
    \left(\sum_{j=0}^{\infty}|a_j|^{q}\|h\|_{L_p}^q\right)^{1/q}
    \sim
    \|\{a_j\}\|_{l_q^{0,1}}.
  \end{split}
  \end{equation}
  Thus, $M_{p,q}\subset F_{p,r}$ implies
  \begin{equation}
    \|\{a_j\}\|_{l_r^{0,1}}\sim \|H\|_{F_{p,r}}\lesssim \|H\|_{M_{p,q}}\sim \|\{a_j\}\|_{l_q^{0,1}},
  \end{equation}
  which implies the desired conclusion.
\end{proof}

We recall a lemma for the local property of $M_{p,q}$, one can see a proof for $1\leq p,q\leq \infty$ in \cite{change of variable}.
The proof for $0< p,q\leq \infty$ is similar and we omit the details here.
\begin{lemma}\label{lemma, localization}
Let $0<p,q\leq \infty$, $r>0$, and $f$ be a tempered distribution supported on $B(x_0, r)$ for some $x_0\in \mathbb{R}^n$.
Then $f\in M_{p,q}$ if and only if $\widehat{f}\in  L_q$. Moreover, we have
\begin{equation}
  \|f\|_{M_{p,q}}\sim \|\widehat{f}\|_{L_q}
\end{equation}
in this case.
\end{lemma}

\begin{proposition}\label{proposition, estiamte of atom}
Let $0<p\leq 1$. We have the following inclusion relation:
\begin{equation}
  F_{p,\infty}^{n(2/p-1)}\subset M_{p,p}.
\end{equation}
\end{proposition}
\begin{proof}
We first verify
\begin{equation}
  \|a\|_{M_{p,p}}\lesssim 1
\end{equation}
for any $n(2/p-1)$-atom $a$.
Take $a$ to be a $n(2/p-1)$-atom as in Definition \ref{definition, atom} (with $s=n(2/p-1)$).
Observing that $|K|\geq [n(2/p-1)]+1\geq [n(1/p-1/2)]+1$, we have
\begin{equation}
  |\partial^{\gamma}a|\leq 1
\end{equation}
for $|\gamma|\leq [n(1/p-1/2)]+1$.
By Bernstein multiplier theorem and Lemma \ref{lemma, localization}, we have the following estimate of $a$,
\begin{equation}
  \begin{split}
    \|a\|_{M_{p,p}}
    \sim
    \|\mathscr{F}^{-1}f\|_{L_p}
    \lesssim
    \sum_{|\gamma|\leq [n(1/p-1/2)]+1}\|\partial^{\gamma}a\|_{L_2}\lesssim 1.
  \end{split}
\end{equation}
Next, we turn to the estimate of $(n(2/p-1),p,\infty)$-atom for $F_{p,\infty}^{n(2/p-1)}$.
By the definition \ref{definition, atom}, a $(n(2/p-1),p,\infty)$-atom $g$ can be represented by
\begin{equation}
  g=\sum_{(\mu,l)<(\nu,k)}d_{\mu l}a_{\mu l}(x)\hspace{10mm}(\text{convergence\ in\ }F_{p,\infty}^{n(2/p-1)})
\end{equation}
for some $\nu\in \mathbb{N}$ and $k\in \mathbb{Z}^n$, where $a_{\mu l}$ are $(Q_{\mu l},n(2/p-1),p,\infty)$-atoms and $d_{\mu l}$ are complex numbers with
     \begin{equation}
       \sup_{(\mu,l)<(\nu,k)}|d_{\mu l}|\leq |Q_{\nu k}|^{-1/p}.
     \end{equation}
For a fixed $\tau\leq \nu$, we denote
\begin{equation}
  g_{\tau}=\sum_{(\tau,l)<(\nu,k)}d_{\tau l}a_{\tau l}(x).
\end{equation}
Then, $g$ can be represented by
\begin{equation}
  g=\sum_{\tau\geq \nu}g_{\tau}\hspace{10mm}(\text{convergence\ in\ }F_{p,\infty}^{n(2/p-1)}).
\end{equation}
We now concentrate ourselves on the estimate of $g_{\tau}$.
By the definition \ref{definition, atom}, we have
\begin{equation}\label{for proof, theorem 2, 1}
  |\partial^{\gamma}a_{\tau l}|\leq |Q_{\tau l}|^{(2/p-1)-|\gamma|/n}=|2^{-\tau n}|^{(2/p-1)-|\gamma|/n}
\end{equation}
for $|\gamma|\leq [n(1/p-1/2)]+1\leq K$.
Recalling $\textbf{supp}a_{\tau l}\subset 5Q_{\tau l}$,
we use (\ref{for proof, theorem 2, 1}) and the almost orthogonality of $a_{\tau l}$ to deduce that
\begin{equation}
  \begin{split}
    |\partial^{\gamma}g_{\tau}|
    = &
    |\sum_{(\tau,l)<(\nu,k)}d_{\tau l}\partial^{\gamma}a_{\tau l}(x)|
    \\
    \lesssim &
    \sup_{(\tau,l)<(\nu,k)}|d_{\mu l}||2^{-\tau n}|^{(2/p-1)-\gamma/n}
    \\
    \lesssim &
    |Q_{\nu k}|^{-1/p}|2^{-\tau n}|^{(2/p-1)-\gamma/n}
  \end{split}
\end{equation}
for all $|\gamma|\leq [n(1/p-1/2)]+1$.
By Bernstein multiplier theorem and Lemma \ref{lemma, localization}, we deduce that
\begin{equation}
  \begin{split}
    \|g_{\tau}\|_{M_{p,p}}
    \sim &
    \|\mathscr{F}^{-1}g_{\tau}\|_{L_p}
    \\
    \lesssim &
    \sum_{|\gamma|\leq [n(1/p-1/2)]+1}\|\partial^{\gamma}g_{\tau}\|_{L_2}
    \\
    \lesssim &
    \sum_{|\gamma|\leq [n(1/p-1/2)]+1}|Q_{\nu k}|^{1/2-1/p}|2^{-\tau n}|^{(2/p-1)-\gamma/n}.
  \end{split}
\end{equation}
By a dilation argument, we have
\begin{equation}
 \|g_{\tau}\|_{M_{p,p}}\sim \|\mathscr{F}^{-1}g_{\tau}\|_{L_p}\lesssim |Q_{\nu k}|^{1/2-1/p}|2^{-\tau n}|^{1/p-1/2}.
\end{equation}
Thus,
\begin{equation}
  \begin{split}
    \|g\|^p_{M_{p,p}}
    =
    \|\sum_{\tau\geq \nu}g_{\tau}\|^p_{M_{p,p}}
    \leq &
    \sum_{\tau\geq \nu}\|g_{\tau}\|^p_{M_{p,p}}
    \\
    \leq &
    |Q_{\nu k}|^{p/2-1}\sum_{\tau\geq \nu}|2^{-\tau n}|^{1-p/2}
    \\
    \lesssim &
    |Q_{\nu k}|^{p/2-1}\cdot |2^{-\nu n}|^{1-p/2}\sim 1.
  \end{split}
\end{equation}
By Lemma \ref{lemma, atom}, we have
\begin{equation}
  \begin{split}
    \|f\|_{M_{p,p}}
    = &
    \|\sum_{j=1}^{\infty}(\mu_ja_j+\lambda_jg_j)\|_{M_{p,p}}
    \\
    \leq &
    \left(\|\sum_{j=1}^{\infty}(\mu_j^p\|a_j\|^p_{M_{p,p}}+\lambda_j^p\|g_j\|^p_{M_{p,p}})\right)^{1/p}
    \\
    \lesssim &
    \left(\|\sum_{j=1}^{\infty}\mu_j^p+\lambda_j^p\right)^{1/p}\lesssim \|f\|_{F_{p,\infty}^{n(2/p-1)}},
  \end{split}
\end{equation}
which is the desired conclusion.
\end{proof}

\medskip

Now we are in the position to prove Theorems 1.1 and 1.2.

\medskip

\hspace{-12pt}{\it Proof of Theorem 1.1.}  We divide this proof into two parts.

\medskip
\textbf{(Sufficiency):} \quad  For $1/p\leq 1/q$,  we have $1/r\leq 1/q$, and then $F_{p,q}\subset F_{p,r}$.
Using Lemma \ref{lemma, modulation and Besov spaces}, we obtain $M_{p,q}\subset B_{p,q}$.
Thus, we deduce that
\begin{equation}
  M_{p,q}\subset B_{p,q}\subset F_{p,q}\subset F_{p,r},
\end{equation}
which is the desired conclusion.

For $1/p> 1/q$, use $M_{p,p}\subset F_{p,p}$ obtained above to deduce that
\begin{equation}
  M_{p,q}^{n(1/p-1/q)+2\epsilon}\subset M_{p,p}^{\epsilon}\subset F_{p,p}^{\epsilon}\subset F_{p,r}
\end{equation}
for any $\epsilon>0$, $r\in (0, \infty]$.

\medskip

{\bf (Necessity):}\quad For $1/p\leq 1/q$, using Proposition \ref{proposition, necessity, 1} and Lemma \ref{lemma, Sharpness of embedding, for uniform decomposition},
we deduce that $l_q^{s,0}\subset l_p^{0,0}$, which implies $s\geq 0$.

On the other hand, if $M_{p,q}\subset F_{p,r}$ holds, we use Proposition \ref{proposition, necessity, special} and Lemma \ref{lemma, Sharpness of embedding, for uniform decomposition} to deduce that $1/r\leq 1/q$.

For $1/p> 1/q$, we use Proposition \ref{proposition, necessity, 1} and Lemma \ref{lemma, Sharpness of embedding, for dyadic decomposition} to deduce
$l_q^{s,0}\subset l_p^{0,0}$, which implies $s> n(1/p-1/q)$.$\hfill\Box$

\bigskip

\hspace{-12pt}{\it Proof of Theorem 1.2.} We divide this proof into two parts.

\medskip
\textbf{(Sufficiency):}  For $1/q\leq 1/p$, by Lemma \ref{lemma, modulation and Besov spaces}, we obtain $ B_{p,\infty} \subset M_{p,\infty}^{n(1-1/p)}$.
Using the fact $F_{p,\infty}\subset B_{p,\infty}$, we deduce that
\begin{equation}\label{for proof, theorem 2, 2}
  F_{p,\infty}\subset M_{p,\infty}^{n(1-1/p)}.
\end{equation}

In addition, we have $F_{p,\infty}^{n(2/p-1)}\subset M_{p,p}$ by Proposition \ref{proposition, estiamte of atom}.
By potential lifting, we obtain
\begin{equation}\label{for proof, theorem 2, 3}
  F_{p,\infty}\subset M_{p,p}^{n(1-2/p)}.
\end{equation}
Thus, the desired conclusion can be deduced by a standard interpolation argument between (\ref{for proof, theorem 2, 2}) and (\ref{for proof, theorem 2, 3}).

For $1/p> 1/q$, recalling $F_{p,\infty}\subset M_{p,p}^{n(1-2/p)}$ obtained in Proposition \ref{proposition, estiamte of atom},
we deduce that
\begin{equation}
  F_{p,r}\subset F_{p,\infty}\subset M_{p,p}^{n(1-2/p)}\subset M_{p,q}^{n(1-1/p-1/q)-\epsilon}
\end{equation}
for any $\epsilon>0$, $r\in (0, \infty]$.
\medskip

\textbf{(Necessity):} We use Proposition \ref{proposition, necessity, 1} to deduce that inclusion relation $l_p^{n(1-1/p),1}\subset l_q^{s+n/q,1}$.
Then, Lemma \ref{lemma, Sharpness of embedding, for dyadic decomposition} yields that $s\leq n(1-1/p-1/q)$ for $1/q\le 1/p$, while the inequality is strict for $1/q>1/p$. $\hfill\Box$

\bigskip

\hspace{-14pt}{\Large\bf An open problem and its difficulties}: Inspired by the article \cite{sobolev and modulation}, one may ask a natural question: can we improve and extend the embedding results in \cite{sobolev and modulation}
  to the more general frame of Triebel-Lizorkin spaces? More exactly, can we establish the sharp conditions for the following two embedding relations:
  \begin{equation}
    (i)\quad M_{p,q}^{s} \subset F_{p,r},\hspace{10mm} (ii)\quad F_{p,r} \subset M_{p,q}^{s}
  \end{equation}
 for $1<p<\infty$, $0<r,q\leq \infty$, $s\in \mathbb{R}$? A key problem is how to find the optimal $r$ for the embedding relations $F_{p,r} \subset M_{p,q}^{n(1-1/p-1/q)}$ in the set
   $P=:\{(p,q): 1<p<2$, $1-1/p<1/q\leq 1/p\}$.
  Unfortunately, due to the following two reasons, it is quite hard to establish such conditions.
  \begin{enumerate}
     \item
     {\bf The ``useless" of complex interpolation:}  Interpolation argument is quite useful in $L_p (p>1)$ case as in \cite{sobolev and modulation},
     while the situation become more complicated in the framework of Triebel-Lizorkin space.
     In this paper, we obtain
     \begin{equation}\label{for remark, 1}
       F_{p,\infty} \subset M_{p,q}^{n(1-1/p-1/q)},\hspace{6mm} p\leq 1, 0\leq 1/q\leq 1/p.
     \end{equation}
      Obviously, $r=\infty$ is optimal in this case.
     However, when $p=q=2$, the optimal $r$ is $2$ rather than $\infty$. In fact, in the embedding relation
     \begin{equation}\label{for remark, 2}
       F_{2,2}\subset M_{2,2},
     \end{equation}
      $r=2$ is optimal.
     If we apply complex interpolation between (\ref{for remark, 1}) and (\ref{for remark, 2}), we only have
     \begin{equation}\label{for remark, 3}
       F_{p,p/(p-1)} \subset M_{p,q}^{n(1-1/p-1/q)}, \hspace{6mm}(p,q)\in P
     \end{equation}
     Now, a new question is: how to determine whether $r=p/(p-1)$ is the optimal one in above embedding relation (\ref{for remark, 3})?
     If $r=p/(p-1)$ is not the optimal one, we have to establish the inclusion relation $F_{p,r} \subset M_{p,q}^{n(1-1/p-1/q)}$ for some $r>p/(p-1)$, $p,q\in P$ directly (without the help of interpolation).
     However,  we face new difficulties as follows.
     \item
    {\bf The lack of appropriate atoms:} In this paper, we use an atom decomposition for inhomogeneous Triebel-Lizorkin spaces (see Lemma \ref{lemma, atom}).
     However, this atom decomposition only works for $0<p\leq 1<q\leq \infty$. When we handle the case $p>1$, the lack of appropriate atoms leads to new difficulties.
  \end{enumerate}


\end{document}